\documentclass[a4paper,12pt]{amsart}
\usepackage{amsfonts}
\usepackage{amssymb}
\usepackage{ifthen}
\usepackage{graphicx}
\usepackage{color}
\nonstopmode \numberwithin{equation}{section}
\graphicspath{ {./gamma_m/} }
\usepackage[rightcaption]{sidecap}

\newtheorem{thm}{Theorem}
\newtheorem{lem}{Lemma}
\newtheorem{cor}{Corollary}[section]

\newtheorem{cl}{Claim}[section]
\newtheorem{ca}{Case}
\newtheorem{sca}{Subcase}
\newtheorem{scl}{Subclaim}
\newtheorem{conj}{Conjecture}

\theoremstyle{definition}
\newtheorem{defn}{Definition}
\newtheorem{example}{Example}[section]
\newtheorem{op}[equation]{Open Problem}
\newtheorem{ques}[equation]{Question}
\newtheorem{rem}{Remark}[section]
\newtheorem{exam}[equation]{Example}

\newcounter {own}
\def\theown {\thesection       .\arabic{own}}

\newenvironment{pf}[1][]{%
 \vskip 3mm
 \noindent
 \ifthenelse{\equal{#1}{}}%
  {{\slshape Proof. }}%
  {{\slshape #1.} }%
 }%
{\qed\bigskip}

\newcounter{alphabet}
\newcounter{tmp}
\newenvironment{Thm}[1][]{\refstepcounter{alphabet}%
\bigskip%
\noindent%
{\bf Theorem \Alph{alphabet}}%
\ifthenelse{\equal{#1}{}}{}{ (#1)}%
{\bf .} \itshape}{\vskip 8pt}

\makeatletter
\newcommand{\Ref}[1]{\@ifundefined{r@#1}{}{\setcounter{tmp}{\ref{#1}}\Alph{tmp}}}
\makeatother

\newcommand{\eit}{{e^{i\theta}}}

\def\be{\begin{equation}}
\def\ee{\end{equation}}

\newcommand{\bee}{\begin{enumerate}}
\newcommand{\eee}{\end{enumerate}}

\newcommand{\blem}{\begin{lem}}
\newcommand{\elem}{\end{lem}}
\newcommand{\bthm}{\begin{thm}}
\newcommand{\ethm}{\end{thm}}
\newcommand{\bcor}{\begin{cor}}
\newcommand{\ecor}{\end{cor}}
\newcommand{\beg}{\begin{exam}}
\newcommand{\eeg}{\end{exam}}
\newcommand{\begs}{\begin{examples}}
\newcommand{\eegs}{\end{examples}}
\newcommand{\bdefe}{\begin{defn}}
\newcommand{\edefe}{\end{defn}}
\newcommand{\bprob}{\begin{prob}}
\newcommand{\eprob}{\end{prob}}
\newcommand{\bques}{\begin{ques}}
\newcommand{\eques}{\end{ques}}
\newcommand{\bei}{\begin{itemize}}
\newcommand{\eei}{\end{itemize}}
\newcommand{\bcon}{\begin{conj}}
\newcommand{\econ}{\end{conj}}
\newcommand{\bop}{\begin{op}}
\newcommand{\eop}{\end{op}}

\newcommand{\bca}{\begin{ca}}
\newcommand{\eca}{\end{ca}}
\newcommand{\bsca}{\begin{sca}}
\newcommand{\esca}{\end{sca}}

\newcommand{\bcl}{\begin{cl}}
\newcommand{\ecl}{\end{cl}}

\newcommand{\bscl}{\begin{scl}}
\newcommand{\escl}{\end{scl}}

\newcommand{\bcons}{\begin{conjs}}
\newcommand{\econs}{\end{conjs}}
\newcommand{\bprop}{\begin{propo}}
\newcommand{\eprop}{\end{propo}}
\newcommand{\br}{\begin{rem}}
\newcommand{\er}{\end{rem}}
\newcommand{\brs}{\begin{rems}}
\newcommand{\ers}{\end{rems}}
\newcommand{\bo}{\begin{obser}}
\newcommand{\eo}{\end{obser}}
\newcommand{\bos}{\begin{obsers}}
\newcommand{\eos}{\end{obsers}}
\newcommand{\bpf}{\begin{pf}}
\newcommand{\epf}{\end{pf}}
\newcommand{\ba}{\begin{array}}
\newcommand{\ea}{\end{array}}
\newcommand{\beq}{\begin{eqnarray}}
\newcommand{\beqq}{\begin{eqnarray*}}
\newcommand{\eeq}{\end{eqnarray}}
\newcommand{\eeqq}{\end{eqnarray*}}

\newcounter{minutes}
\divide\time by 60
\newcounter{hours}
\multiply\time by 60 \addtocounter{minutes}{-\time}

\begin{document}

\bibliographystyle{amsplain}
\title[Boundary behavior results for harmonic mappings]{Koebe and Carathe\'odory type boundary behavior results for harmonic mappings}
\def\thefootnote{}
\footnotetext{ \texttt{\tiny File:~\jobname .tex,
          printed: \number\day-\number\month-\number\year,
          \thehours.\ifnum\theminutes<10{0}\fi\theminutes}
} \makeatletter\def\thefootnote{\@arabic\c@footnote}\makeatother

\author{Daoud Bshouty}
\address{Daoud Bshouty, Department of Mathematics, Technion -- Israel Institute of Technology, Haifa 32000, Israel}
\email{daoud@technion.ac.il}

\author{Jiaolong Chen}
\address{Jiaolong Chen, Key Laboratory of High Performance Computing and Stochastic Information Processing (HPCSIP) (Ministry of Education of China), College of Mathematics and Computer Science,
Hunan Normal University, Changsha, Hunan 410081, People's Republic of China}
\email{jiaolongchen@sina.com}

\author{Stavros Evdoridis}
\address{Stavros Evdoridis, Department of Mathematics and Systems Analysis, Aalto University, P. O. Box 11100, FI-00076 Aalto,
 Finland.} \email{stavros.evdoridis@aalto.fi}

\author{Antti Rasila$^\dagger$}
\address{Antti Rasila, Technion -- Israel Institute of Technology,  Guangdong Technion,
241 Daxue Road, Shantou, Guangdong 515063, People's Republic of China} \email{antti.rasila@iki.fi; antti.rasila@gtiit.edu.cn}

\thanks{The research was partly supported by NNSF of China (No. 11801166, No. 11971124), Academy of Finland, Natural Science Foundation of Hunan Province (No. 2018JJ3327)
and the construct program of the key discipline in Hunan Province.}

\keywords{Harmonic  mapping, Boundary behavior.\\
$^\dagger$ {\tt
 {\tt Corresponding author.}
} }


\begin{abstract}
We study the behavior of the boundary function of a harmonic mapping from global and local points of view. Results related to the Koebe lemma are proved, as well as a generalization of a boundary behavior theorem by Bshouty, Lyzzaik and Weitsman. We also discuss this result from a different point of view, from which a relation between the boundary behavior of the dilatation at a boundary point and the continuity of the boundary function of our mapping can be seen. 
\end{abstract}

\maketitle
\pagestyle{myheadings}
\markboth{D. Bshouty, J. Chen, S. Evdoridis, and A. Rasila}{Boundary behavior results for harmonic mappings}

\section{Introduction}
In this paper we consider certain questions related to the boundary behavior of harmonic mappings. More precisely, we first give  results  analogous to classical Koebe's lemma. 
In the last part, we present a new approach and a generalized version of \cite[Theorem 6]{BLW}.

{In 1915, Koebe}
\cite{Ko} proved that if a bounded analytic function tends to zero on a sequence of arcs in the unit disk $\mathbb{D}$, which approach a subarc of the unit circle $\partial \mathbb{D}$, then it must be identically zero.
In \cite{Run}, {Rung}
studied the behavior of the analytic functions when the arcs are of positive hyperbolic diameter. Results related to the Koebe lemma for quasiconformal mappings can also be found in \cite{Ra}. Our first aim is to obtain a similar result for harmonic mappings.

Another related result is the following classical theorem of Lindel\"of {(\cite[page 259]{Ru}): }

\begin{Thm}
If $\gamma$ is a parametric curve in $[0,1]$, such that $|\gamma (t)|<1$ and $\gamma (1)=1$, then for every analytic function $f$ which is bounded in $\mathbb{D}$, with
$$
\lim_{t\to 1}f(\gamma(t))
{=\alpha,}
$$
its angular limit at $1$ exists and is equal to $\alpha$.
\end{Thm}
{For generalized versions of this result for different classes of mappings, see e.g. \cite{Ge} and \cite{Ra}.}
Clearly, the harmonic function $u(z)=\arg z$ shows that the result as such does not hold for harmonic functions. However, in \cite{PR}, Ponnusamy and Rasila showed a connection between the multiplicity of the zeros of a harmonic mapping, as they tend to a boundary point along a line, and the existence of the angular limit at this point. These results show that under certain additional assumptions, the boundary behavior can be controlled, which serves as a motivation for this investigation.

Secondly, we discuss the Carath\'eodory-Osgood-Taylor theorem (also known as Carath\'eodory's extension theorem for conformal mappings), which was proved {by Carath\'eodory and independently by Osgood and Taylor in 1913.}
It allows the extension of a conformal mapping, obtained for example from the Riemann Mapping Theorem, to the boundary of the domain.

\begin{Thm}\label{Thm-A}
Let $\Omega$ be a Jordan domain and $f$ be a conformal mapping from $\mathbb{D}$ onto $\Omega$. Then $f$ can be extended to a homeomorphism of $\overline{\mathbb{D}}$ onto $\overline{\Omega}$ (the closure of $\Omega$).
\end{Thm}

An analogue of the above theorem also holds for the quasiconformal mappings on the plane
{(see \cite[pages 42-44]{LV}).}
This shows that conformal and quasiconformal mappings
{of $\mathbb{D}$ onto
Jordan domains behave similarly at $\partial\mathbb{D}$.}

However, the harmonic case is more complicated, because such a result does not hold. We have the following well-known counterexample.  Recall that for a function $f\colon \mathbb{D} \to  \mathbb{C}$ the {\it cluster set} of $f$ at the point $e^{it}$, is the quantity
$$C(f,e^{it})=\bigcap_{U} \overline{f(U\cap \mathbb{D})},$$
where the intersection is taken over all neighborhoods $U$ of $e^{it}$. 

\begin{example}
Let $f$ be the function given by the Poisson formula
$$f(z)=\frac{1}{2\pi}\int_0^{2\pi}\frac{1-|z|^2}{|e^{it}-z|^2}e^{i\theta (t)}dt, \qquad z\in\mathbb{D},$$
where $$\theta (t)=
\left\{
  \begin{array}{ll}
     0, \hspace{9mm} 0\leq t \leq 2\pi/3, & \hbox{\;} \\
     2\pi/3,  \hspace{2mm} 2\pi/3\leq t\leq 4\pi/3, & \hbox{\;}\\
     4\pi/3,  \hspace{2mm} 4\pi/3\leq t <2\pi .& \hbox{\;}
  \end{array}
\right.
$$
Then $f$ is a univalent harmonic mapping from $\mathbb{D}$ onto the triangle with vertices at the points $1$, $e^{2\pi i/3}$, $e^{4\pi i/3}$. In addition, it can be shown that,
{as $z$ approaches $\partial\mathbb{D}$, }
the cluster sets of $f$ at the points $1$, $e^{2\pi i/3}$ and $e^{4\pi i/3}$ are the three sides of the triangle and the circular arcs (separated by those points) are mapped to the three vertices. Clearly, the boundary function of $f$ is not a homeomorphism.
\end{example}

\subsection*{Continuity on the boundary}
The above discussions lead to the question of continuity of the boundary function. {Recall the next result of Hengartner and Schober}   \cite{HS}  which is concerned with the behavior of the boundary function of a harmonic mapping in $\mathbb{D}$.

\begin{Thm}
Let $\Omega$ be a bounded, simply connected Jordan domain whose boundary is locally connected. Suppose that $a$ is analytic
{in $\mathbb{D}$}
 with $a(\mathbb{D})\subset \mathbb{D}$ and $w_0$ is a fixed point in $\Omega$. Then there exists a univalent function $f$ such that $\overline{f_{\overline{z}}}(z)=a(z)f_z(z)$, i.e. $a(z)=a_f(z)$, with the following properties\\
\begin{enumerate}
\item[(a)] $f(0)=w_0$, $f_z(0)>0$ and $f(\mathbb{D})\subset \Omega$.
\item[(b)] There is a countable set $E\subset \partial \mathbb{D}$ such that the unrestricted limits $f^*(e^{it})=\lim_{z\to e^{it}}f(z)$ exist on $\partial \mathbb{D}\setminus E$ and they are on $\partial \Omega$.
\item[(c)] The functions
$$f^*(e^{it-})=\operatorname{ess}\lim_{s\uparrow t}f^*(e^{is}) \text{  and  }  f^*(e^{it+})=\operatorname{ess}\lim_{s\downarrow t}f^*(e^{is})$$
exist on $\partial\mathbb{D}$ and are equal on $\partial\mathbb{D}\setminus E$.
\item[(d)] The cluster set of $f$ at $e^{it}\in E$ is the straight line segment joining $f^*(e^{it-})$ to $f^*(e^{it+})$.
\end{enumerate}
\end{Thm}
For $e^{it}\in E$, we say that $f^*$ has a $\textit{jump discontinuity}$ at $e^{it}$. Otherwise $f^*$ is said to be $\textit{continuous}$.
{When $\|a(z)\|_{\infty}=1$, then $E$ can be nonempty,
where $\|a(z)\|_{\infty}=\sup\{|a(z)|:z\in\mathbb{D}\}$.}
 If $|a(z)|<k<1$, then $f$ is quasiconformal in $\mathbb{D}$ and hence it can be extended homeomorphically from $\overline{\mathbb{D}}$ onto $\overline{\Omega}$.

In the last part of the paper, we present a generalized version of a theorem proved by Bshouty, Lyzzaik and Weitsman (see \cite[Theorem 6]{BLW}). There, we define Jordan curves in $\mathbb{D}$, which end at a point $\zeta \in \partial \mathbb{D}$. Then we show that a condition on the integrals of a quantity which contains the second complex dilatation of a harmonic mapping $f$, over these arcs, imply the continuity of the boundary function $f^*$ at $\zeta$. This observation gives an extra tool that can be used in order for us to examine the behavior of the boundary function of a harmonic mapping. Moreover, it can be considered as a Carath\'eodory type result for a class of non-quasiconformal harmonic mappings.

\section{Preliminaroes}\label{sec-1}

A continuous mapping $f=u+iv$ is
a {\it complex-valued harmonic} mapping in a domain $\Omega\subset \mathbb{C}$,
if both $u$ and $v$ are real harmonic functions in $\Omega$. Then we write $\Delta f=0$, where $\Delta$ is the Laplace operator
$$\Delta=\frac{\partial^2}{\partial
x^2}+\frac{\partial^2}{\partial y^2}=4\frac{\partial^{2}}{\partial
z\partial\bar{z}}.$$
In any simply connected domain $\Omega$, $f$ has the unique representation $f = h+ \overline{g}$
{with $g(0)=0$,}
where $h$ and $g$ are analytic (see \cite{cl, du}).
A necessary and sufficient condition for $f$ to be locally
univalent and sense-preserving in $\Omega$ is that $|h'(z)|>|g'(z)|$ for all $z\in \Omega$. Let $$D_{f}:=\frac{|f_{z}|+|f_{\overline{z}}|}{|f_{z}|-|f_{\overline{z}}|}=\frac{1+|a_{f}|}{1-|a_{f}|}, \quad
{a_f :=\overline{f_{\overline{z}}}/f_z.}
$$
The quantity $D_{f}:=D_{f}(z)$ is called the {\it dilatation} of $f$ at the point $z$. Clearly, $1\leq D_{f}(z)<\infty$.
If $D_{f}(z)$ is bounded throughout a given region $\Omega$ by a constant $K\in [1,\infty)$, then the sense-preserving diffeomorphism $f$ is said to be {\it $K$-quasiconformal} in $\Omega$.
{A mapping }
that is $K$-quasiconformal for some $K\geq 1$ is simply called {\it quasiconformal}. In addition, the ratio
$\nu_{f}=f_{\overline{z}}/f_{z}$
is called the {\it complex dilatation} of $f$ and $a_f=\overline{f_{\overline{z}}}/{f_z}$ is said to be the {\it second complex dilation} of $f$. Thus, $0\leq |a_{f}|=|\nu_f|<1$ if, and only if, $f$ is sense-preserving.

{Let $\mathbb{D}(z_{0}, r)=\{z\in\mathbb{C}:|z-z_{0}|<r\}$ and $\overline{\mathbb{D}}(z_{0}, r)=\{z\in\mathbb{C}:|z-z_{0}|\leq r\}$.
In particular, we write
$\mathbb{D}(r)=\mathbb{D}(0, r)$ and $\overline{\mathbb{D}}(r)=\overline{\mathbb{D}}(0, r)$.}
In this paper, we consider the harmonic mappings
{in $\mathbb{D}$}
 or in $\mathbb{H}=\{z\in\mathbb{C}:\text{Im}\{z\}>0\}$.
Let $A\subset \mathbb{C}$ be a nonempty set.
We denote the $\textit{diameter}$ of $A$ by $d(A)$,
where $d(A)=\sup\{|z-w|: z,w\in A\}$.
The distance between a point $z_0\in \mathbb{C} \setminus A$ and the set $A$ is given by $d(z_0,A)=\inf\{|z_0-w|:w\in A\}$.


\section{Conformal invariants}\label{sec-2}

{In this section, we recall some}
 useful definitions and results we need for the proofs.

\subsection{Modulus of a paths' family}
A path in $\mathbb{C}$ is a continuous mapping $\gamma:I\rightarrow \mathbb{C}$, where $I$ is a (possibly unbounded) interval in $\mathbb{R}$.

Let $\Gamma$ be a paths' family in $\mathbb{C}$ and $\mathcal{F}(\Gamma)$ be the set of all Borel functions $\rho:\mathbb{C}\rightarrow [0,\infty]$ such that
$$ \int_{\gamma} \rho (z)\, |dz|\geq 1,$$
for every locally rectifiable path $\gamma \in\Gamma$.
The functions in $\mathcal{F}(\Gamma)$ are called {\it admissible functions} for $\Gamma$.
We define the {\it conformal modulus} of $\Gamma$ to be
$$\mathcal{M}(\Gamma)=\inf_{\rho\in \mathcal{F}(\Gamma)} \iint_{\mathbb{C}} \rho ^2(z)\,dxdy,$$
where $z=x+iy$.

Let $\Gamma_{1}$ and $\Gamma_{2}$ be paths' families in $\mathbb{C}$.
We say that $\Gamma_{2}$ is {\it minorized}
by $\Gamma_{1}$ and write $\Gamma_{1}<\Gamma_{2}$ if every $\gamma \in \Gamma_{2}$ has a subpath in $\Gamma_{1}$.
If
$\Gamma_{1}<\Gamma_{2}$,
{then $\mathcal{F}(\Gamma_{1})\subset \mathcal{F}(\Gamma_{2})$,}
 and hence $\mathcal{M}(\Gamma_{1})\geq \mathcal{M}(\Gamma_{2})$. Suppose that $A$ and $B$ are contained in a domain $D$. We write $\Delta (A,B;D)$ to denote the family of curves, connecting $A$ and $B$ in $D$. When $D$ is the whole plane, we write $\Delta (A,B)$.
For more information about the modulus of a paths' family, see e.g. \cite{AVV, Vas, Vuo}.

A domain $D$ in $\overline{\mathbb{C}}=\mathbb{C}\cup \{\infty\}$ is called a ring, if $\overline{\mathbb{C}}\backslash D $ has exactly two components.
If the components are $E$ and $F$, we denote the ring by $R(E,F)$.

\begin{example}
We consider the annulus $D:=A(R,R')=\{z=r\eit : R<r<R', \, 0\leq \theta < 2\pi\}$. Then, $$\mathcal{M}\left(\Delta(R,R',D)\right) = \frac{2\pi}{\log{\frac{R'}{R}}}.$$
 \end{example}

By \cite[Theorem 34.3]{Vai}, we have the following result.

\begin{lem}\label{lem-A} A sense-preserving homeomorphism $f:D\rightarrow D'$ is $K$-quasiconformal if and only if
$$ \mathcal{M}(\Gamma)/K\leq \mathcal{M}(f(\Gamma))\leq K\mathcal{M}(\Gamma)$$
for every path family $\Gamma$ in $D$.
\end{lem}

The lemma given below is a useful tool for us in order to prove Theorem \ref{thm-1} (see \cite[Lemma 5.33]{Ra}).

\begin{lem}\label{lem-B}
Let $C\subset \mathbb{D}$ be a continuum with $0<d(C)\leq 1$. Then, if $d(0,C)>0$, we have that
$$\mathcal{M}\left( \Delta(\overline{\mathbb{D}}(1/2), C;\mathbb{D})\right) \geq \frac{1}{4}\tau_2\left(\frac{d(0,C)}{d(C)}\right) .$$
\end{lem}
\subsection{Canonical ring domains} \label{Crd}
The complementary components of the {\it Gr\"{o}tzsch ring} $R_G(s)$ in $\mathbb{C}$ are $\overline{\mathbb{D}}$ and $[s,\infty]$, $s>1$, and those of the {\it Teichm\"{u}ller ring} $R_T(s)$ are $[-1,0]$ and $[s,\infty]$, $s>0$.
We define two special functions $\gamma_2(s)$, $s>1$ and $\tau_2(s)$, $s>0$ by
$$
\left\{
  \begin{array}{ll}
     \gamma_2(s)=\mathcal{M}\big(\Delta(\overline{\mathbb{D}},[s,\infty])   \big), & \hbox{\;} \\
     \tau_2(s)=\mathcal{M}\big(\Delta([-1,0],[s,\infty])   \big), & \hbox{\;}
  \end{array}
\right.
$$
respectively. We shall refer to those functions as the {\it Gr\"{o}tzsch capacity function}
and the {\it Teichm\"{u}ller capacity function}. It is well-known \cite[Lemma 5.53]{Vuo} that for all
{$s>1$,}
$$ \gamma_2(s)=\frac{1}{2}\tau_2(s^{2}-1),$$
and that
$\tau_2:(0,\infty)\rightarrow (0,\infty)$ is a decreasing homeomorphism.
{Our notation is based on \cite{AVV}.}

\subsection{Hyperbolic distance}
{Let $\Omega$ be a simply connected domain in $\mathbb{C}$.
If $a$, $b\in\Omega$, then the {\it hyperbolic distance} between $a$ and $b$ in $\Omega$ is denoted by $\rho_{\Omega}(a,b)$ (cf. \cite[page 19]{Vuo}).
For $a\in\Omega$ and $M>0$, the hyperbolicball $\{x\in\Omega:\rho_{\Omega}(a,x)<M\}$ is denoted by $D_{\Omega}(a,M)$.
If $z_1, z_2 \in \mathbb{H}$, then (cf. \cite[Equality (2.8)]{Vuo})}
$$\cosh \left( \rho _{\mathbb{H}}(z_1,z_2) \right)=1+\frac{|z_1-z_2|^2}{2\text{Im} \{z_1\}\text{Im} \{z_2\}}.$$

An overview on hyperbolic geometry can be found in \cite{BM}.

\section{Koebe type results}

In this section, we shall show give two Koebe type results for sense-preserving harmonic mappings with different conditions.


The following result is an analogue of Koebe's lemma for univalent harmonic mappings in $\mathbb{D}$.
\begin{thm}\label{thm-1}
Suppose $f:\mathbb{D}\rightarrow \mathbb{C}$ is either a univalent sense-preserving harmonic mapping with $f_{\overline{z}}(0)=0$ or constant.
 Let $\alpha \in\mathbb{C} $ and $\{C_{j}\}$ be a sequence of nondegenerate continua such that
 $C_{j}\subseteq \mathbb{D}(r_{j})$,
{$C_{j}\rightarrow \partial \mathbb{D}$ as $j$ goes to infinity,
and $|f(z)-\alpha|<M_{j}$ for $z\in C_{j}$,
where $\lim_{j\rightarrow\infty}r_j=1$ and $\lim_{j\rightarrow\infty}M_j=0$.}
If
$$ \limsup_{j\rightarrow \infty} \tau_2\left( \frac{1}{d(C_{j})}\right)\left(\log \frac{1}{M_{j}}\right)\frac{1-r_{j}}{1+r_{j}}=\infty,$$
where $\tau_2$ is the Teichm\"uller capacity function, defined in Section \ref{Crd}, then $f \equiv \alpha$ in $\mathbb{D}$. \\

In particular, if
$$ \limsup_{j\rightarrow \infty}  \left(\log  \frac{1}{d(C_{j})}\right)^{-1}\left(\log \frac{1}{M_{j}}\right)\frac{1-r_{j}}{1+r_{j}}=\infty,$$
then $f\equiv \alpha$.
\end{thm}

Let $f(z)$ be analytic at $z_0$ and suppose that $f(z_0)=0$, but $f(z)$ is not identically zero. Then $f(z)$ has
a zero of order or multiplicity $n$ at $z_0$ if
$$
f(z_0)=f'(z_0)=\cdots=f^{(n-1)}(z_0)=0,\text{ and } f^{(n)}(z_0)\neq 0.
$$
If $f(z)$ is analytic at $z_0$, and has zero of order $n$ at $z_0$, we write $\mu(z_0,f)=n$. For a sense-preserving harmonic
mapping $f$ in ${\mathbb D}$ the multiplicity can be obtained from the decomposition $f= h + \overline{g}$.
Suppose that $h$ and $g$ have respectively multiplicity $n$ and $m$ at $0$ with $n\leq m$. We say that $f$ has zero of order $n$ at $z_0$ and write $\mu(z_0,f)=n$.

Next, we state a result concerning the boundary behavior of a sense-preserving harmonic mapping, defined in the upper half-plane $\mathbb{H}$, which has a condition involving multiplicity of the zeroes.

\begin{thm}\label{thm-2}
Suppose that $f=h+\overline{g}$ is a sense-preserving harmonic mapping in {$\mathbb{H}$}
with $|f(z)|<1$. Let $\{b_k\}$ be a sequence of points in $\mathbb{H}$ such that $f(b_k)=0$ and
{$\lim_{k\to \infty }b_k=0$.}
If
$$
{\lim_{ k\to \infty}\left( \frac{10-\rm{Im} \{ b_k\} }{10+\rm{Im} \{ b_k\}} \right)^{\mu(b_k,f)} = 0,}
$$
where $\mu (b_k,f)$ denotes the multiplicity of the zero of the harmonic mapping at $b_k$, then $f\equiv 0$ in $\mathbb{H}$.
\end{thm}

\subsection{Proof of Theorem \ref{thm-1}}
Assume that $f$ is univalent in $\mathbb{D}$ and  $0<d(C_{j})\leq \frac{1}{4}$.
Let $\Gamma_{j}=\Delta \big(\overline{\mathbb{D}}(\frac{1}{2}),C_{j};\mathbb{D}(r_{j}) \big)$. Then by Lemma \ref{lem-B}, we have
\be\label{eq-3.1}
\mathcal{M}(\Gamma_{j})\geq \frac{1}{4}\tau_2\left( \frac{d(0,C_{j})}{d(C_{j})} \right)\geq \frac{1}{4}\tau_2\left( \frac{1}{d(C_{j})} \right).
\ee
Now, let $w=d\left(f\left(\overline{\mathbb{D}}(\frac{1}{2})\right),\alpha \right)>0$.
{Because $\lim_{j\rightarrow\infty}M_j=0$,}
without loss of generality, we may assume that $M_{j}<\min\{w^{2},1/w\}$.
Then,
$$
f(\Gamma_{j} )>
{\Delta\big(\partial\mathbb{D}(\alpha, M_{j}),\partial\mathbb{D}(\alpha, w);}
\mathbb{D}(\alpha, w)
\backslash \mathbb{D}(\alpha, M_{j})\big),
$$
and hence,
\be\label{eq-3.2}\mathcal{M}(f(\Gamma_{j}))\leq 2\pi\left(\log \frac{w}{M_{j}} \right)^{-1}\leq 2\pi\left(\frac{1}{2}\log \frac{1}{M_{j}} \right)^{-1},\ee
where the last inequality is obtained in \cite[Lemma 5.31]{Ra}.

Because $f$ is a univalent sense-preserving harmonic mapping with $f_{\overline{z}}(0)=0$,
we have that $J_f(z)\:=|f_z(z)|^2-|f_{\overline{z}}(z)|^2>0$ and thus, $a _f$ is analytic in $\mathbb{D}$, with $|a_f|=|\nu _f|<1$ and $a _f(0)=0$. By the classical Schwarz lemma it follows that  $|a_{f}(z)|\leq |z|$ in $\mathbb{D}$ and $$D_{f}=\frac{|f_{z}|+|f_{\overline{z}}|}{|f_{z}|-|f_{\overline{z}}|}=\frac{1+|a_{f}|}{1-|a_{f}|}\leq \frac{1+|z|}{1-|z|}.$$
Therefore, for any $z\in\mathbb{D}(r_{j})$,  $$D_{f}\leq \frac{1+|z|}{1-|z|}\leq \frac{1+r_{j}}{1-r_{j}}=K_{j},$$
which implies that $f$ is $K_{j}$-quasiconformal in the disk $\mathbb{D}(r_{j})$. Then Lemma \ref{lem-A} leads to
\be\label{eq-3.3} \mathcal{M}(\Gamma_{j} )\leq  K_{j} \mathcal{M}(f(\Gamma_{j})).\ee
It follows from \eqref{eq-3.1}$\sim$\eqref{eq-3.3}, we have
$$\tau_2\left( \frac{1}{d(C_{j})} \right)\leq 16\pi\frac{1+r_{j}}{1-r_{j}}\left(\log \frac{1}{M_{j}} \right)^{-1}. $$
That is
$$\tau_2\left( \frac{1}{d(C_{j})} \right)\left(\log \frac{1}{M_{j}} \right)\frac{1-r_{j}}{1+r_{j}}\leq 16\pi,$$
which is a contradiction with the assumption. The proof of the first part is finished.

By the same argument as in the second part of \cite[Theorem 5.34]{Ra}, we have that
$$ \tau_2\left( \frac{1}{d(C_{j})} \right)\left(\log \frac{1}{M_{j}} \right)\geq \frac{\pi}{2} \left(\log \frac{1}{d(C_{j})} \right)^{-1}\left(\log \frac{1}{M_{j}} \right),  $$
which gives the second part of the theorem.  \qed

Recall the following, which is originally from an unpublished manuscript of  Mateljevi\'c and Vuorinen (see \cite[Lemma 4.3]{PR}):

\begin{lem}\label{hschwarz}
Let $f$ be a sense-preserving harmonic mapping of $ \mathbb{D}$ such that $f(0)=0$ and $f(\mathbb{D})\subset  \mathbb{D}$. Then
\begin{equation*}
|f(z)|\le \frac{4}{\pi}\arctan |z|^{\mu(0,f)}\le \frac{4}{\pi}|z|^{\mu(0,f)}
\quad\mbox{ for $z\in \mathbb{D}$.}
\end{equation*}
\end{lem}

\subsection{Proof of Theorem \ref{thm-2}}
Assume that $f:\mathbb{H}\to \mathbb{D}$ is a sense-preserving harmonic mapping with $f(b_k)=0$.
{Let $\phi_k : \mathbb{D}\to \mathbb{H}$ be the M\"obius transformation, }
with $\phi_k (0)=b_k$, $k=1,2,\ldots$.
Then, the function $\psi_k = f\circ \phi_k $ is a sense-preserving harmonic mapping
{from $\mathbb{D}$ into itself, }
with $\psi_k(0)=0$.
{Therefore, from Lemma \ref{hschwarz}, }
we have that
$$|\psi_k (w)|=|(f\circ \phi_k)(w)| \leq \frac{4}{\pi}\arctan|w|^{\mu(0,\psi_k)} \leq  \frac{4}{\pi}|w|^{\mu(b_k,f)}\, , \quad w\in \mathbb{D},$$
where $\mu(z_0,f)$ denotes the multiplicity of the zero of the harmonic mapping $f$ at the point $z_0$, as it is stated in \cite{PR}.

{For $w\in \mathbb{D}$, let $z=\phi_k(w)$.}
Because the M\"obius transformations from $\mathbb{D}$ onto $\mathbb{H}$ are isometries of the hyperbolic distance,
then $z\in D_{\mathbb{H}}(b_k,R_k)$ if and only if $w\in D_{\mathbb{D}}(0,R_k)$,
which is equivalent to the condition $|w|<\tanh\frac{R_k}{2}$, where $R_k>0$.
Thus, for all $z\in D_{\mathbb{H}}(b_k,R_k)$, we have
\begin{eqnarray}
\label{eq1}
|f(z)|=|\psi_k(w)|\leq \frac{4}{\pi}\left( \tanh\frac{R_k}{2}\right)^{\mu(b_k,f)}.
\end{eqnarray}
\bcl\label{claim-4.1} There exists an index $k_0$ such that for $k\ge k_0$, $i\in D_{\mathbb{H}}(b_k,R_k)$, where
{$R_k=\log{\frac{10}{\rm{Im} \{b_k\}}}$.}
\ecl

It follows from the assumption $\lim_{k\rightarrow\infty}b_k=0$
{in the theorem}
that there exists an index $k_0$, such that, for $k\ge k_0$ the sequence $\{b_k\}$ is contained {in $\mathbb{D}(i,2)$.}
Hence, for every $k\geq k_0$,
$$\cosh\big(\rho_{\mathbb{H}}(i,b_k)\big)=1+\frac{|i-b_k|^2}{2 \text{Im} \{b_k\}}\leq 1+\frac{2}{\text{Im} \{b_k\}}.$$
Because the number of the terms of $\{b_k\}$  which are outside of {$\mathbb{D}(i,2)$ is finite, }
we may assume that every term is contained in it.
Therefore,
$$\frac{e^{\rho_{\mathbb{H}}(i,b_k)}+e^{-\rho_{\mathbb{H}}(i,b_k)}}{2}=\cosh\left(\rho_{\mathbb{H}}(i,b_k)\right)\leq 1+\frac{2}{\text{Im} \{b_k\}}$$
or
\be\label{eq-4.5}\log\left(e^{2\rho_{\mathbb{H}}(i,b_k)}+1\right)\leq \log\left(e^{\rho_{\mathbb{H}}(i,b_k)}\left(2+\frac{4}{\text{Im} \{b_k\}}\right) \right).\ee
Because $2\rho_{\mathbb{H}}(i,b_k)<\log\left(e^{2\rho_{\mathbb{H}}(i,b_k)}+1\right)$
and
\begin{eqnarray*}
2+\frac{4}{\text{Im} \{b_k\}}<\frac{10}{\text{Im} \{b_k\}},
\quad \text{ for }\text{Im} \{b_k\}<3, \\
\end{eqnarray*}
\eqref{eq-4.5} implies that
$$\rho_{\mathbb{H}}(i,b_k)< \log{\frac{10}{\text{Im} \{b_k\}}}.$$
Hence, Claim \ref{claim-4.1} is true.

Now, it follows from \eqref{eq1} and Claim \ref{claim-4.1} that
$$|f(i)|\leq \frac{4}{\pi}\left( \tanh\frac{\log{\frac{10}{\text{Im} \{b_k\}}}}{2}\right)^{\mu(b_k,f)}.$$
Moreover, $$\tanh\left(\frac{\log{\frac{10}{\text{Im} \{b_k\}}}}{2}\right)=\frac{10-\text{Im} \{b_k\}}{10+\text{Im} \{b_k\}}$$
and thus, our initial assumption implies that
$$|f(i)|\leq \lim_{k\to \infty}\left( \frac{4}{\pi}\left( \tanh\frac{\log{\frac{10}{\text{Im} \{b_k\}}}}{2}\right)^{\mu(b_k,f)}\right) =0.$$
The same argument can be applied at any point in the disk $D_{\mathbb{H}}(b_k,R_k)$ and hence, $f(z)=0$ for all $z\in D_{\mathbb{H}}(b_k,R_k)$.

{Because the line segment $[i/2 ,i]$}
is contained in each disk $D_{\mathbb{H}}(b_k,R_k)$ for sufficiently large $k$, by applying the Uniqueness (or Identity) Theorem at both the analytic and the anti-analytic parts of $f$, we obtain that $f\equiv 0$ in $\mathbb{H}$.
 \qed

 \begin{rem}
Theorem \ref{thm-2} stays true even if we replace the constant $10$ by $4+\epsilon$, for any number $\epsilon>0$. As $\{b_k\}$ tends to zero, almost every term of the sequence is contained in the strip $\{z\in \mathbb{C}: 0<\text{Im} \{z\}<\epsilon/2\}$, for a fixed number $\epsilon>0$. Considering the terms $b_k$ of the disk $\mathbb{D}(i,2)$, with $\text{Im} \{b_k\}<\epsilon/2$, leads to the improved condition
$$\left( \frac{(4+\epsilon)-\text{Im} \{b_k\}}{(4+\epsilon)+\text{Im} \{b_k\}} \right)^{\mu(b_k,f)} \to 0\qquad \text{as} \qquad k\to \infty ,$$
 at the statement of the theorem.
 \end{rem}


\section{Dilatation and the boundary function} \label{dilat}
The next result is a generalization of a result by Bshouty, Lyzzaik and Weitsman (see \cite{BLW}), which removes the assumption that the dilatation has to be a Blaschke product. It connects the boundary behavior of the dilatation
$a_{f}$ of a univalent harmonic mapping $f$
with the continuity of its boundary function $f^*$ at a point of $\partial \mathbb{D}$. Here $f^*$ is a complex Lebesgue integrable function on $\partial  \mathbb{D}$, which has $f$ as its Poisson integral.

Let $f(z)=h(z)+\overline{g(z)}$ be a harmonic mapping, then its dilatation {is $a_{f}(z)=\frac{g'(z)}{h'(z)}.$}
One cannot expect that the dilatation itself can decide
{whether $f^*(\zeta)$ is continuous without involving both $h$ and $g$ themselves, where $\zeta\in \partial \mathbb D$ and $f^*$ is the boundary function of $f$.}
If we assume $f$ is sense preserving,
{then $|a_{f}(z)|<1$}
 which ensures that $|g'(z)|<|h'(z)|.$ This
may reduce the problem to $h(z).$ Indeed this is the case, Lemma 2 in [3]  states in particular:

If $f$ is a sense-preserving harmonic mapping and $f^*$ is of bounded variation then for $\zeta\in \partial \mathbb D,$
$f^*$ is continuous at $\zeta$ if, and only if,
$$\lim_{r\rightarrow1^-}(1-r)h'(r\zeta)=0.$$
In the next theorem, we allow ourselves to connect the continuity of $f^*$ with $a$ by tying them  to area of $f(\mathbb D).$

For $m>0$ and $\zeta =e^{i\theta _0} \in \partial \mathbb{D}$, we define the curve
$\Gamma_{\zeta , m}(\theta)=(1 -m|\theta -\theta _0|)e^{i\theta}$,
$0<|\theta - \theta _0| \leq \min\{\pi, 1/m\}$.
It is clear that $\Gamma_{\zeta , m}$ is completely included in $\mathbb{D}$,
tending to $\zeta $ as $\theta \to \theta_0$.

In the case $\zeta = 1$, we have the curve $\Gamma _m (\theta)= (1-m|\theta|)e^{i\theta}$,
which is symmetric with respect to the real axis, meeting it either at the point $m\pi -1$ (when $m<1/\pi$) or at 0 (when $m\geq 1/\pi$).
We denote by $R_m$ the bounded region which has the boundary $\Gamma_m$.

\begin{thm}
\label{theorem3}
Suppose that $\Omega$ is a bounded, convex and simply connected domain. Let $a_f$ be an analytic function
{from $\mathbb{D}$ into}
 itself and $f$ be the univalent solution of $\overline{f_{\overline{z}}}=a_f f_z$
 {from $\mathbb{D}$ onto $\Omega$.}
 If for each $0<m<1/\pi$
$$
{\mathcal L (m):=}
\int_{\Gamma_{\zeta,m}} \frac{1-|a_f(z)|^2}{1-|z|^2}|dz|=\infty, $$
then the boundary function $f^*$ is continuous at $\zeta \in \partial \mathbb{D}$.
\end{thm}

\begin{proof}
Without loss of generality, we can assume that $\zeta=1$.
Suppose that $\mathcal L (m)=\infty$,
 for every $m$, $0<m<1/\pi$,
and assume on the contrary that $f^*$ has a jump at $1$.
Then \cite[Theorem 3]{BLW} implies that the angular limit $\lim_{z\to 1}(1-|z|)|h'(z)|$ is equal to $c>0$.
Because $|a_f(z)|<1$ and $f$ is univalent in $\mathbb{D}$, we have that $|h'(z)|>0$ in $\mathbb{D}$ and the same is true for $(1-|z|)|h'(z)|$.
Indeed, $|h'(z)|>|g'(z)|$ means that both $|h'(z)|$ and $(1-|z|)|h'(z)|$ do not vanish in $\mathbb{D}$.
The fact that the angular limit $\lim_{z\to 1}(1-|z|)|h'(z)|$ is positive implies that $|h'(z)|$ is also positive when it tends to $1$ in any angle.

Fix $m=m_0<1/\pi$, \ $0<|\theta|\leq \pi$ and $z\in \Gamma_{m_{0}} (\theta)$, we have
$$\Big| \frac{dz}{d\theta}\Big| =\sqrt{m^{2}_{0}+(1-m_0|\theta|)^{2}}\leq m_{0}+1-m_0|\theta|<2.$$
Thus,
\begin{eqnarray}
\infty =
\mathcal L (m_0)
&=& \int_{\Gamma _{m_0}} \frac{(|h'(z)|^2-|g'(z)|^2)(1-|z|)}{|h'(z)|^2(1-|z|)^2(1+|z|)}|dz| \nonumber \\
&\leq &
\int_{\Gamma _{m_0}} c'  (|h'(z)|^2-|g'(z)|^2)m_{0}|\theta|\; |dz|
 \nonumber \\
&<&
c_{1}
 \int_{-\pi}^\pi (|h'(z)|^2-|g'(z)|^2) |\theta| d\theta \label{prop1} ,
\end{eqnarray}
where $c'$, $c_{1}$ are two positive constants and
{$z= (1-m_{0}|\theta|)e^{i\theta}$.
}

\begin{figure}[h]
\begin{center}
\includegraphics[width=7.5cm]{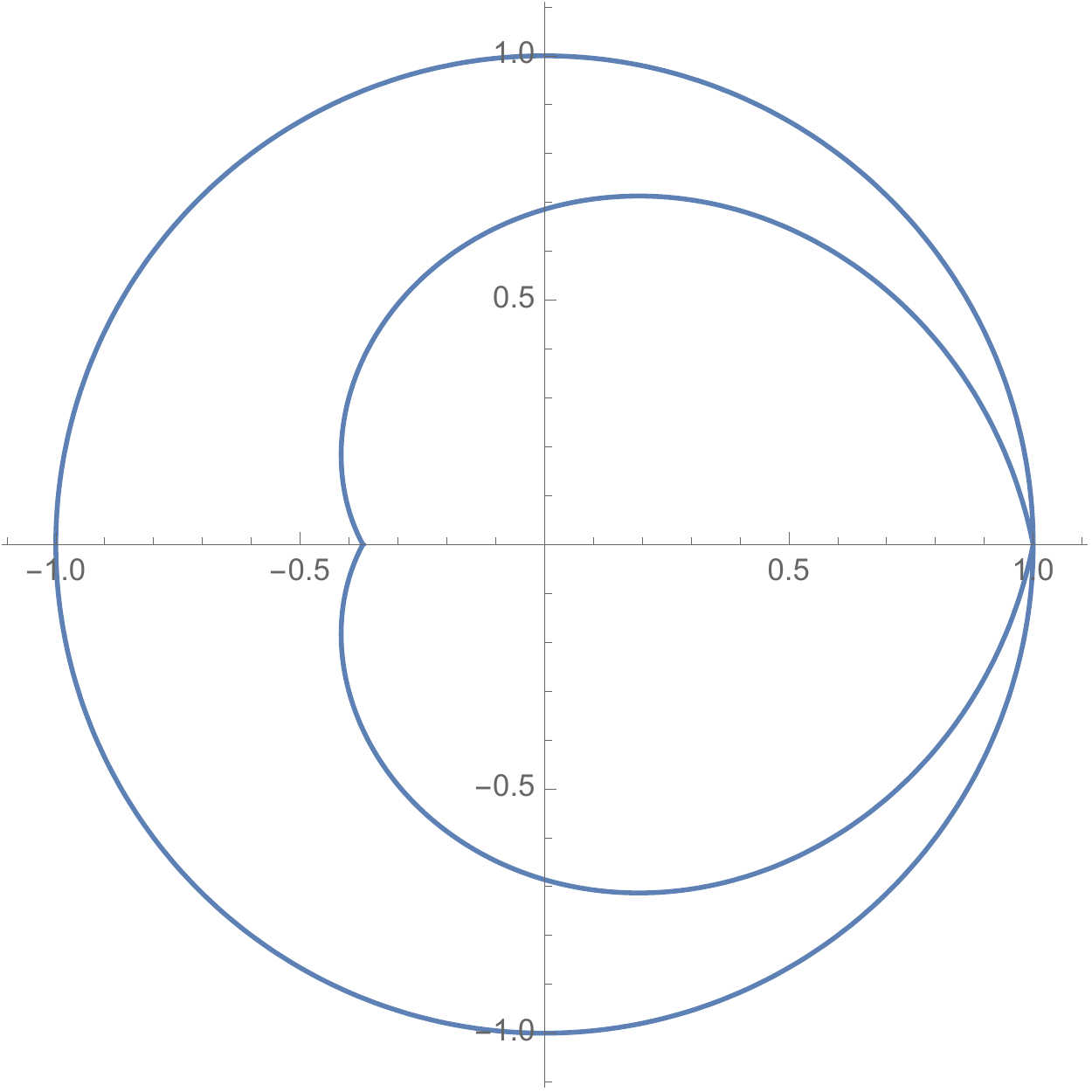}
\end{center}

\caption{The curve $\Gamma_m$ for $m=0.2$ and $| \theta | \in (0, \pi]$, contained in $\mathbb{D}$.}\label{gamma-fig}
\end{figure}

In the final step of the proof we will study the behavior of the curve $\Gamma_m$ (see Figure \ref{gamma-fig})
as a mapping from $R_{m_0}=[(0,m_0)\times (-\pi,\pi)]\setminus \{(m,0), 0<m<m_0\} $
onto the region $Q_{m_0}=\left( \mathbb{D} \setminus R_{m_0}\right) \setminus (-1, -1/2)$.
Observe that $\Gamma_m$ tends to $\partial \mathbb{D}$ as $m\to 0$. In order to make it clear, we write
$$\Gamma _m(\theta) = (1-m|\theta|)e^{i\theta}=re^{i\theta}.$$
Then $r=(1-m|\theta|)$  and the Jacobian $\left|\frac{D(r,\theta)}{D(m,\theta)}\right|=|\theta|$. Therefore, for the area $A$ of $f(\mathbb{D})$ we have that
\begin{eqnarray*}
A &=& \iint_{\mathbb{D}}(|h'(z)|^2 - |g'(z)|^2)\,dxdy > \iint_{R_{m_{0}}}  (|h'(z)|^2 - |g'(z)|^2) r\hspace{0.3mm}dr\hspace{0.3mm} d\theta \\
&=& \iint_{R_{m_{0}}} (|h'(z)|^2 - |g'(z)|^2)(1-m|\theta |)|\theta|\hspace{0.3mm}dm \hspace{0.3mm} d\theta \\
&>&\int_0^{m_0} (1-m \pi )\left( \int_{-\pi}^{\pi} (|h'(z)|^2 - |g'(z)|^2) |\theta|   d\theta \right)\hspace{0.3mm} dm = \infty,
\end{eqnarray*}
where the last equality follows from \eqref{prop1}.
This contradicts the initial assumption that $f$ is bounded, and hence, the theorem is proved.
\end{proof}

\begin{rem}
Let $\Omega $ be as above, $a_f:\mathbb{D} \to \mathbb{D}$ analytic
and $f$ be the solution of the equation $\overline{f_{\overline{z}}} = a_{f} f_{z}$
from $\mathbb{D}$ onto $\Omega$. If there exists an analytic function $F:\mathbb{D} \to \mathbb{D}$ such that $$\int_{\Gamma_{m,\zeta}} \frac{1-|F(z)|^2}{1-|z|^2} |dz|=\infty ,$$ for any $m>0$, and $a_f$ is majorized by $F$ (denoted by $a_f\ll F$) in $\mathbb{D}$,
then the boundary function of $f$ is continuous at $\zeta \in \mathbb{D}$.
\end{rem}
Indeed, majorization implies that there exists an analytic function $\phi :\mathbb{D} \to \mathbb{D}$ such that $a_f (z)=\phi (z)F(z) $ for all $z\in \mathbb{D}$.
Hence,
\begin{eqnarray*}
\int_{\Gamma_{m,\zeta}} \frac{1-|a_f(z)|^2}{1-|z|^2} |dz| &=& \int_{\Gamma_{m,\zeta}} \frac{1-|\phi(z)F(z)|^2}{1-|z|^2} |dz| \\
&\geq & \int_{\Gamma_{m,\zeta}} \frac{1-|F(z)|^2}{1-|z|^2} |dz| = \infty .
\end{eqnarray*}
The previous result concludes the proof.

\begin{example}
{If $a_{f}(z)=\alpha z$ with $|\alpha|<1$, then $\mathcal L (m)$ is infinite at every point, and hence,}
 the boundary function $f^*$ is a continuous function.
When $|\alpha |=1$ then $\mathcal L(m)$ is finite, in one case when $\Omega$ is a quadrilateral triangle, $f^*$ has jump continuity between every  two vertices and in another case when $\Omega$ is a strictly concave triangle, {\it i.e.} a triangle with three strictly concave sides with respect to its interior and three cusps of zero angles,  $f^*$ is continuous at every point.
\end{example}

Theorem \ref{theorem3} states that if $\mathcal L(m)=\infty$ for all accessible $m$ at some point $\zeta$ and $(1-|z|)|h'(z)|> C$
{with constant $ C>0$,}
then the area of $\Omega$ is infinite. In the next theorem we study the converse theorem.

\begin{thm}
 Suppose that $\Omega$ is a convex and simply connected domain. Let $a_f$ be an analytic function from $\mathbb{D}$ into itself and $f$ be the univalent solution of $\overline{f_{\overline{z}}}=a_f f_z$
 {from $\mathbb{D}$ onto $\Omega$.} 
 Fix $\zeta\in\partial\mathbb D$ where $f(\zeta)$ is finite, and assume that
$$ ( 1-|z|)|h'(z)|\leq H $$
in any compact set in $\mathbb D\cup \{\zeta\}$. Then
$$\mathrm{Area}\big(\Omega)<
{\frac{\pi H^2}{0.15}}
\int_0^{1/\pi} \mathcal L(m)dm.$$
\end{thm}

\begin{rem}
It follows immediately that if $\mathrm{Area}(\Omega)=\infty$, then at least for $m\in E_\zeta$, where $E_\zeta$  a set of positive measure, $\mathcal L(m)=\infty$ at $\zeta.$ 
\end{rem}

\begin{proof}
Without loss of generality we can assume that $\zeta=1.$ Let $0<m<\frac1\pi.$ Then on $\Gamma_m$ we have $\frac{|dz|}{d\theta}=|m +i(1-m|\theta|)|>0.3 $ so that
\begin{eqnarray*}
\mathcal L(m) =\int_{\Gamma _{m}}\frac{1-|a_f(z)|^2}{1-|z|^2}|dz| &=& \int_{\Gamma _{m}} \frac{(|h'(z)|^2-|g'(z)|^2)(1-|z|)}{|h'(z)|^2(1-|z|)^2(1+|z|)}|dz| \nonumber \\
&\geq & \int_{\Gamma _{m}}\frac{  (|h'(z)|^2-|g'(z)|^2)m|\theta |}{2H^2}|dz| \nonumber \\
&>& \frac{0.3m}{2 H^2} \int_{-\pi}^\pi (|h'(z)|^2-|g'(z)|^2) |\theta| d\theta.
\end{eqnarray*}
Thus
$$\int_0^{1/\pi} \mathcal L(m)dm\geq \frac{0.15}{\pi^2H^2}\int_{-\pi}^\pi (|h'(z)|^2-|g'(z)|^2) |\theta| d\theta.$$

On the other hand,
\begin{eqnarray*}
{\mathrm{Area}\big(\Omega)}
&=& \iint_{\mathbb{D}}(|h'(z)|^2 - |g'(z)|^2)\,dxdy = \iint_{\mathbb D  }  (|h'(z)|^2 - |g'(z)|^2) r\hspace{0.3mm}dr\hspace{0.3mm} d\theta \\
&=& \int_m\int_\theta (|h'(z)|^2 - |g'(z)|^2)(1-m|\theta |)|\theta|\hspace{0.3mm}dm \hspace{0.3mm} d\theta \\
&<&
{\int_0^{1/\pi}  \int_{-\pi}^{\pi} (|h'(z)|^2 - |g'(z)|^2) |\theta|   d\theta } \hspace{0.3mm} dm \\
&=&
{\frac{1}{\pi}}
\int_{-\pi}^{\pi} (|h'(z)|^2 - |g'(z)|^2) |\theta|   d\theta \hspace{0.3mm}.
\end{eqnarray*}
Therefore,
$$\mathrm{Area}\big(\Omega)<
\frac{\pi H^2}{0.15}
\int_0^{1/\pi} \mathcal L(m)dm,$$
{completing the proof.}
\end{proof}

\end{document}